\documentclass[10pt]{amsart}

\author{Daniel Thompson}
\usepackage[english]{babel}
\usepackage{amsmath}
\usepackage{amsthm}
\usepackage{amssymb}
\usepackage[backend=bibtex,style=alphabetic]{biblatex}
\usepackage[T1]{fontenc}
\usepackage{enumerate}
\usepackage{microtype}
\usepackage[none]{hyphenat}
\usepackage{setspace}
\usepackage[dvipsnames,table]{xcolor}

\bibliography{biblio.bib}

\newtheorem{theorem}{Theorem}[section]
\newtheorem{corollary}[theorem]{Corollary}
\newtheorem{lemma}[theorem]{Lemma}
\newtheorem{proposition}[theorem]{Proposition}
\theoremstyle{definition}
\newtheorem{remark}[theorem]{Remark}

\newtheorem{example}[theorem]{Example}

\title{Division algebras that generalize Dickson semifields}

\begin{document}
\maketitle
\begin{abstract}
We generalize Knuth's construction of Case I semifields quadratic over a weak nucleus, also known as generalized Dickson semifields, by doubling of central simple algebras. We thus obtain division algebras of dimension $2s^2$ by doubling central division algebras of degree $s$. Results on isomorphisms and automorphisms of these algebras are obtained in certain cases.
\end{abstract}

\section*{Introduction}

The commutative division algebras constructed by Dickson \cite{Dic} yield proper semifields of even dimension over finite fields. They have been subsequently studied in many papers, for example in \cite{Bur}, \cite{Bur2}, \cite{HTW}, \cite{Thompson19}. Knuth recognised that Dickson's commutative division algebras also appear as a special case of another family of semifields \cite{Knu}: A subalgebra $L$ of a division algebra $S$ is called a \textit{weak nucleus} if $x(yz)-(xy)z=0$, whenever two of $x,y,z$ lie in $L$.  Semifields which are quadratic over a weak nucleus are split into two cases; Case I semifields contain Dickson's construction as the only commutative semifields of this type. Due to this, Case I semifields are also called \textit{generalized Dickson semifields}. Their construction is as follows: given a finite field $K=GF(p^n)$ for some odd prime $p$, define a multiplication on $K\oplus K$ by $$(u,v)(x,y)=(uv+c\alpha(v)\beta(y), \sigma(u)y+vx),$$ for some automorphisms $\alpha$, $\beta$, $\sigma$ of $K$ not all the identity automorphism and $c\in K\setminus K^2$. This construction produces a proper semifield containing $p^{2n}$ elements. Further work on semifields quadratic over a weak nucleus was done in \cite{Ganley} and \cite{Cohen}.\\
In this paper, we define a doubling process which generalizes Knuth's construction in \cite{Knu}: for a central simple associative algebra $D/F$ or finite field extension $K/F$, we define a multiplication on the $F$-vector space $D\oplus D$ (resp. $K\oplus K$) as $$(u,v)(x,y)=(ux+c\sigma_1(v)\sigma_2(y),\sigma_3(u)y+v\sigma_4(x))$$ for some $c\in D^{\times}$ and $\sigma_i\in Aut_F(D)$ for $i=1,2,3,4$ (resp. $c\in K^{\times}$ and $\sigma_i\in Aut_F(K)$). This yields an algebra of dimension $2dim_F(D)$ or $2[K:F]$ over $F$. 
Over finite fields, our construction yields examples of some Hughes-Kleinfeld, Knuth and Sandler semifields (for example, see \cite{Cor}) and all generalized Dickson and commutative Dickson semifields \cite{Knu}\cite{Dic}. Hughes-Kleinfeld, Knuth and Sandler semifield constructions were studied over arbitrary base fields in \cite{BrownSteelePump}. Dickson's commutative semifield construction was introduced over finite fields in \cite{Dic} and considered over any base field of characteristic not 2 when $K$ is a finite cyclic extension in \cite{Bur}. This was generalized to a doubling of any finite field extension and central simple algebras in \cite{Thompson19}.\\
After preliminary results and definitions, we define a doubling process for both a central simple algebra $D/F$ and a finite field extension $K/F$; we recover the multiplication used in Knuth's construction of generalized Dickson semifields when $\sigma_4=id$. We find criteria for them to be division algebras. We then determine the nucleus and commutator of these algebras and examine both isomorphisms and automorphisms. The results of this paper are part of the author's PhD thesis written under the supervision of Dr S. Pumpl\"{u}n.

\section{Definitions and preliminary results}

In this paper, let $F$ be a field. We define an $F$-algebra $A$ as a finite dimensional $F$-vector space equipped with a (not necessarily associative) bilinear map $A\times A\to A$ which is the multiplication of the algebra. $A$ is a \textit{division algebra} if for all nonzero $a\in A$ the maps $L_a:A\to A$, $x\mapsto ax$, and $R_a:A\to A$, $x\mapsto xa$, are bijective maps. As $A$ is finite dimensional, $A$ is a division algebra if and only if there are no zero divisors \cite{Sch}.\\
The \textit{associator} of $x,y,z\in A$ is defined to be $[x,y,z]:=(xy)z-x(yz).$ Define the \textit{left, middle and right nuclei} of $A$ as $\text{Nuc}_l(A):=\lbrace x\in A \mid [x,A,A]=0\rbrace,$ $\text{Nuc}_m(A):=\lbrace x\in A \mid [A,x,A]=0\rbrace,$ and $\text{Nuc}_r(A):=\lbrace x\in A \mid [A,A,x]=0\rbrace.$ The left, middle and right nuclei are associative subalgebras of $A$. Their intersection $\text{Nuc}(A):=\lbrace x\in A \mid [x,A,A]=[A,x,A]=[A,A,x]=0\rbrace$ is the \textit{nucleus} of $A$. The \textit{commutator} of $A$ is the set of elements which commute with every other element, $Comm(A):=\lbrace x\in A\mid xy=yx \:\forall y\in A\rbrace.$ The \textit{center} of $A$ is given by the intersection of $\text{Nuc}(A)$ and $Comm(A)$, $Z(A):=\lbrace x\in \text{Nuc}(A)\mid xy=yx\: \forall y\in A\rbrace.$ For two algebras $A$ and $B$, any isomorphism $f:A\to B$ maps $\text{Nuc}_l(A)$ isomorphically onto $\text{Nuc}_l(B)$ (similarly for the middle and right nuclei).\\
An algebra $A$ is \textit{unital} if there exists an element $1_A\in A$ such that $x1_A=1_Ax=x$ for all $x\in A$. 
%We say $u$ is a \textit{left unit} if $ux=x$ for all $x\in A$ and a \textit{right unit} if $xu=x$ for all $x\in A$.\\
%An \textit{involution} on an algebra $A$ over field $F$ is a map $\sigma: A\rightarrow A$ such that \begin{itemize}
%\item $\sigma(x+y)=\sigma(x)+\sigma(y)$,
%\item $\sigma(xy)=\sigma(y)\sigma(x)$,
%\item $\sigma^2(x)=x$,
%\end{itemize}
%for all $x,y\in A$. Restricting $\sigma$ to $F$ gives an automorphism of $F$ which is either the identity or of order 2. Then $\sigma$ is called an involution \textit{of the first kind} or \textit{of the second kind}, respectively.\\
%A \textit{d-linear form} over $F$ is a $F$-multilinear map $\theta:A\times...\times A\to F$ ($d$ copies) which is \textit{symmetric}, i.e. $\theta(x_1,x_2,...,x_d)$ is invariant under all permutations of its variables. A \textit{form of degree d} over $F$ is a map $N:A\to F$ such that $N(ax)=a^dN(x)$ for all $a\in F$, $x\in A$ and such that the map $\theta:A\times...\times A\to F$ defined by $$\theta(x_1,x_2,...,x_d)=\frac{1}{d!}\sum_{1\leq i_1<...<i_l\leq d}(-1)^{d-1}N(x_{i_1}+...+x_{i_l})$$ ($1\leq l\leq d$) is a $d$-linear form over $F$. 
A form $N:A \to F$ is called \textit{multiplicative} if $N(xy)=N(x)N(y)$ for all $x,y\in A$ and \textit{nondegenerate} if we have $N(x)=0$ if and only if $x=0$. Note that if $N:A\to F$ is a nondegenerate multiplicative form and $A$ is a unital algebra, it follows that $N(1_A)=1_F$. Every central simple algebra admits a uniquely determined nondegenerate multiplicative form, called the \textit{norm} of the algebra.\\

\section{A doubling process which generalizes Knuth's construction}\label{GCDD_Field}

Let $D$ be a central simple associative division algebra over $F$ with nondegenerate multiplicative norm form $N_{D/F}:D\to F$. Given $\sigma_i\in Aut_F(D)$ for $i=1,2,3,4$ and $c\in D^{\times}$, define a multiplication on the $F$-vector space $D\oplus D$ by $$(u,v)(x,y)=(ux+c\sigma_1(v)\sigma_2(y),\sigma_3(u)y+v\sigma_4(x)).$$ We denote the $F$-vector space endowed with this multiplication by $\text{Cay}(D,c,\sigma_1,\sigma_2,\sigma_3,\sigma_4)$. We can also define an analogous multiplication on $K\oplus K$ for a finite field extension $K/F$ for some $c\in K^{\times}$ and $\sigma_i\in Aut_F(K)$. We similarly denote these algebras by $\text{Cay}(K,c,\sigma_1,\sigma_2,\sigma_3,\sigma_4).$ This yields unital $F$-algebras of dimension $2dim_F(D)$ and $2[K:F]$ respectively. When $\sigma_4=id$, our multiplication is identical to the one used in the construction of generalized Dickson semifields. For every subalgebra $E\subset D$ such that $c\in E^{\times}$ and $\sigma_i\mid_E=\phi_i\in Aut_F(E)$ for $i=1,2,3,4$, it is clear that $\text{Cay}(E,c,\phi_1,\phi_2,\phi_3,\phi_4)$ is a subalgebra of $\text{Cay}(D,c,\sigma_1,\sigma_2,\sigma_3,\sigma_4)$.

\begin{theorem}\label{KDivision Algebras} 
\begin{enumerate}[(i)]
\item If $N_{D/F}(c)\not\in N_{D/F}(D^{\times})^2$, $\text{Cay}(D,c,\sigma_1,\sigma_2,\sigma_3,\sigma_4)$ is a division algebra.
\item If $K$ is separable over $F$ and $N_{K/F}(c)\not\in N_{K/F}(K^{\times})^2$, then $\text{Cay}(K,c,\sigma_1,\sigma_2,\sigma_3,\sigma_4)$ is a division algebra.
\end{enumerate}
\end{theorem}
\begin{proof}

(i) Suppose $(0,0)=(u,v)(x,y)$ for some $u,v,x,y\in D$ such that $(u,v)\neq(0,0)\neq(x,y)$. This is equivalent to \begin{align}
ux+c\sigma_1(v)\sigma_2(y)=&0,\label{Div1K}\\
\sigma_3(u)y+v\sigma_4(x)=&0.\label{Div2K}
\end{align}
Assume $y=0$. Then by (\ref{Div1K}), $ux=0$, so $u=0$ or $x=0$ as $D$ is a division algebra. As $(x,y)\neq (0,0)$, we must have $x\neq 0$ so $u=0$. Then by (\ref{Div2K}), $v\sigma_4(x)=0$ which implies $v=0 \mbox{ or } x=0$. This is a contradiction, thus it follows that $y\neq 0$. By (\ref{Div2K}), $v\sigma_4(x)=-\sigma_3(u)y.$ Let $N=N_{D/F}:D\to F$. Taking norms of both sides, we have \begin{align*}
&N(v)N(x)=-N(u)N(y)\\
\implies &N(u)= -N(v)N(x)N(y)^{-1},
\end{align*}
since $y\neq 0$. Substituting this result into (\ref{Div1K}) implies \begin{align}
0=&N(u)N(x)+N(c)N(v)N(y)\nonumber\\
=&(-N(v)N(x)N(y)^{-1})N(x)+N(c)N(v)N(y) \nonumber\\
=&N(v)[(N(x)N(y)^{-1})^2-N(c)]. \label{Div3K}
\end{align}
If $N(v)=0$, then $v=0$ so by (\ref{Div1K}) $ux=0$ implies $x=0$ (else $(u,v)=(0,0)$).  Thus (\ref{Div3K}) implies $N(c)=0\not\in F^{\times},$ which cannot happen as $c\neq 0$.
Thus we must have $N(v)\neq 0$ and $(N(x)N(y)^{-1})^2=N(c)$. Thus we conclude $N(c)\in N(D^{\times})^2$.\\
(ii) The proof follows analogously as in (i); we require $K$ to be separable over $F$ so that $N_{K/F}(\sigma(x))=N_{K/F}(x)$ for all $\sigma\in Aut_F(K)$ and $x\in K$.
\end{proof}

\begin{remark}
If $F=\mathbb{F}_{p^s}$ and $K=\mathbb{F}_{p^r}$ is a finite extension of $F$, then $Aut_F(K)$ is cyclic of order $r/s$ and is generated by $\phi^s$, where $\phi$ is defined by the Frobenius automorphism $\phi(x)=x^p$ for all $x\in K$. Then $A=\text{Cay}(K,c,\sigma_1,\sigma_2,\sigma_3,\sigma_4)$ is a division algebra if and only if $c$ is not a square in $K$. The proof of this is analogous to the one given in \cite[p. 53]{Knu}.
\end{remark}

\subsection{Commutator and nuclei}

%We would like to know when we obtain commutative or associative algebras with our constructions. In order to have a measure of how commutative or associative our algebras are, we look at the commutator and nuclei of the algebras, respectively.
Unless otherwise stated, we will write $A_D=\text{Cay}(D,c,\sigma_1,\sigma_2,\sigma_3,\sigma_4)$ and $A_K=\text{Cay}(K,c,\sigma_1,\sigma_2,\sigma_3,\sigma_4)$.

\begin{proposition}\label{Commutator of Cay(K)} If $\sigma_1=\sigma_2$ and $\sigma_3=\sigma_4$, $Comm(A_D)=F\oplus F$ and $A_K$ is commutative. Otherwise, $Comm(A_D)=F$ and $Comm(A_K)=\lbrace (u,0)\mid \sigma_3(u)=\sigma_4(u)\rbrace\subseteq K.$
\end{proposition}
\begin{proof}
We compute this only for $A_D$ as the computations for $A_K$ follow analogously. By definition, $(u,v)\in Comm(A_D)$ if and only if for all $x,y\in D$, $(u,v)(x,y)=(x,y)(u,v).$ This is equivalent to \begin{align*}
ux+c\sigma_1(v)\sigma_2(y)=&xu+c\sigma_1(y)\sigma_2(v),\\
\sigma_3(u)y+v\sigma_4(x)=&\sigma_3(x)v+y\sigma_4(u),
\end{align*}
for all $x,y\in D$. The first equation implies $u\in F$ and either $\sigma_1=\sigma_2$ or $v=0$ . Additionally, the second equation implies $v\in F$ and $\sigma_3=\sigma_4$ or $v=0$. The result follows immediately.
\end{proof}

\begin{proposition}\label{Cay(B) left nuc}
\begin{enumerate}[(i)]
\item Suppose that at least one of the following holds:
\begin{itemize}
\item $\sigma_2\circ\sigma_4\neq id$,
\item $\sigma_1\circ\sigma_4\neq \sigma_2\circ\sigma_3$,
\item $\sigma_4\circ\sigma_1\neq \sigma_3\circ\sigma_2$.
\end{itemize}
Then $\text{Nuc}_l(A_D)=\lbrace x\in D\mid \sigma_1\circ\sigma_3(x)=c^{-1}xc\rbrace\subset D$ and $\text{Nuc}_l(A_K)=\text{Fix}(\sigma_1\circ\sigma_3)\subset K$.

\item Suppose that at least one of the following holds:
\begin{itemize}
\item there exists some $x\in D$ (resp. $K$) such that $\sigma_1\circ\sigma_3(x)\neq c^{-1}xc$,
\item $\sigma_2\circ\sigma_4\neq id$,
\item for all $v\in D$, there exists some $x\in D$ (resp. $K$) such that $\sigma_3(c)\sigma_3\circ\sigma_1(x)\sigma_3\circ\sigma_2(v)\neq x\sigma_4(c)\sigma_4\circ\sigma_1(v)$.
\end{itemize}
Then $\text{Nuc}_m(A)=\text{Fix}(\sigma_3^{-1}\circ\sigma_2^{-1}\circ\sigma_1\circ\sigma_4)$ for both $A=A_D$ and $A=A_K$.
%\text{Nuc}_m(A_D)=\lbrace x\in D\mid \sigma_1\circ\sigma_4(x)=\sigma_2\circ\sigma_3(x)\rbrace\subset D$.

\item Suppose that at least one of the following holds:
\begin{itemize}
\item  there exists some $x\in D$ (resp. $K$) such that $\sigma_1\circ\sigma_3(x)\neq c^{-1}xc$,
\item $\sigma_1\circ\sigma_4\neq\sigma_2\circ\sigma_3$,
\item there exists some $x\in D$ (resp. $K$) such that $\sigma_3(c)\sigma_3\circ\sigma_1(x)\neq x\sigma_4(c)$.
\end{itemize}
Then $\text{Nuc}_r(A)=\text{Fix}(\sigma_2\circ\sigma_4)$ for both $A=A_D$ and $A=A_K$. 
%$=\lbrace x\in D\mid \sigma_2\circ\sigma_4(x)=x\rbrace\subset D$ and $\text{Nuc}_r(A_K)=\text{Fix}(\sigma_2\circ\sigma_4)\subset K$.
\end{enumerate}
\end{proposition}
\begin{proof}
We show the proof for (i) since (ii) and (iii) follow analagously. First consider all elements of the form $(k,0)$ for $k\in D$. Then $(k,0)\in \text{Nuc}_l(A_D)$ if and only if we have $((k,0)(u,v))(x,y)=(k,0)((u,v)(x,y))$ for all $u,v,x,y\in D$. Computing this directly, we obtain the equations \begin{align*}
kux+c\sigma_1(\sigma_3(k)v)\sigma_2(y)=&kux+kc\sigma_1(v)\sigma_2(y),\\
\sigma_3(ku)y+\sigma_3(k)v\sigma_4(x)=&\sigma_3(k)\sigma_3(u)y+\sigma_3(k)v\sigma_4(x).
\end{align*}
These hold for all $u,v,x,y\in D$ if and only if $c\sigma_1\circ\sigma_3(k)=kc$, i.e. we have $\sigma_1\circ\sigma_3(k)=c^{-1}kc$. The same calculations yield that this holds for all $u,v,x,y\in D$ if and only if $\sigma_1\circ\sigma_3(k)=k$.\\
To show that there are no other elements in the left nucleus, it suffices to check that there are no elements of the form $(0,m)$, $m\in D$, in $\text{Nuc}_l(A_D)$. This is because the associator is linear in the first component:
$[(k,m),(u,v),(x,y)]=[(k,0),(u,v),(x,y)]+[(0,m),(u,v),(x,y)].$ If $(0,m)\in \text{Nuc}_l(A_D)$, then for all $u,v,x,y\in D$ we have $((0,m)(u,v))(x,y)=(0,m)((u,v)(x,y)).$ This holds for all $u,v,x,y\in D$ if and only if $$c\sigma_1(m)[\sigma_2(v)x+\sigma_1(\sigma_4(u))\sigma_2(y)]=c\sigma_1(m)[\sigma_2(v)\sigma_2(\sigma_4(x))+\sigma_2(\sigma_3(u))\sigma_2(y)],$$ $$m\sigma_3(c\sigma_2(v))y=m\sigma_4(c\sigma_1(v)\sigma_2(y)).$$ In order for this to be satisfied for all $u,v,x,y\in D$, we have either $m=0$ or all the following must hold:\begin{itemize}
\item $\sigma_2\circ\sigma_4=id$,
\item $\sigma_1\circ\sigma_4=\sigma_2\circ\sigma_3$,
\item $\sigma_4\circ\sigma_1=\sigma_3\circ\sigma_2$.
\end{itemize}
If $m\neq 0$, this contradicts the assumptions we made, so this yields $m=0$. The same argument also gives $m=0$ in the field case.
\end{proof}

\begin{corollary} $A_K$ is associative if and only if $A_K=\text{Cay}(K,c,\sigma,\tau,\sigma^{-1},\tau^{-1})$ for some $\tau,\sigma\in Aut_F(K)$ such that $(\sigma\circ\tau)^2=id$ and $c\in \text{Fix}(\sigma\circ\tau)$.
\end{corollary}

%\begin{corollary} If $A_K$ is commutative, then $\text{Nuc}_l(A)=\text{Fix}(\sigma_1\circ\sigma_3)$ if $\sigma_1\neq \sigma_3^{-1}$ and $\text{Nuc}(A)=A$ if $\sigma_1=\sigma_3^{-1}$.
%\end{corollary}

As the center of $A$ is defined as $Z(A)=Comm(A)\cap \text{Nuc}_l(A)\cap \text{Nuc}_m(A)\cap \text{Nuc}_r(A),$ we see that $Z(A_K)\subset K$ unless $\sigma_1=\sigma_2=\sigma$ and $\sigma_3=\sigma_4=\sigma^{-1}$. If $A_K=\text{Cay}(K,c,\sigma,\sigma,\sigma^{-1},\sigma^{-1})$ for some $\sigma\in Aut_F(K)$, then $A_K$ is a commutative, associative algebra.

\subsection{Isomorphisms}

\begin{theorem}\label{ConstructingIsomorphismsDtoD'}
Let $D$ and $D'$ be two central simple $F$-algebras (respectively, $K$ and $L$ finite field extensions of $F$) and $g,h:D\to D'$ be two $F$-algebra isomorphisms. Let $A_D=\text{Cay}(D,c,\sigma_1,\sigma_2,\sigma_3,\sigma_4)$ and $B_{D'}=\text{Cay}(D',g(c)b^2,\phi_1,\phi_2,\phi_3,\phi_4)$ for some $b\in F^{\times}$ (resp. $A_K=\text{Cay}(K,c,\sigma_1,\sigma_2,\sigma_3,\sigma_4)$ and $B_L=\text{Cay}(L,g(c)\phi_1(b)\phi_2(b),\phi_1,\phi_2,\phi_3,\phi_4)$ for some $b\in K^{\times}$). If 
\begin{align}
\phi_i&=g\circ\sigma_i\circ h^{-1} \mbox{ for } i=1,2, \label{Equation 1}\\
\phi_i&=h\circ\sigma_i\circ g^{-1} \mbox{ for } i=3,4, \label{Equation 2}
\end{align}
then the map $G:A\to B,$  $G(u,v)=(g(u),h(v)b^{-1})$ defines an $F$-algebra isomorphism.
\end{theorem}
\begin{proof}
We show the proof in the central simple algebra case. It follows analogously when we take field extensions $K$ and $L$. Clearly $G$ is $F$-linear, additive and bijective. It only remains to show that $G$ is multiplicative; that is, $G((u,v)(x,y))=G(u,v)G(x,y)$ for all $u,v,x,y\in D$. First we have \begin{align*}
G(u,v)G(x,y)=&(g(u),h(v)b^{-1})(g(x),h(y)b^{-1})\\
=&(g(u)g(x)+g(c)b^2\phi_1(h(v)b^{-1})\phi_2(h(y)b^{-1}),\phi_3(g(u))h(y)b^{-1}+h(v)b^{-1}\phi_4(g(x)))\\
=&(g(ux)+g(c)\phi_1(h(v))\phi_2(h(y)),[\phi_3(g(u))h(y)+h(v)\phi_4(g(x))]b^{-1}).
\end{align*}
It similarly follows that \begin{align*}
G((u,v)(x,y))=&G(ux+c\sigma_1(v)\sigma_2(y),\sigma_3(u)y+v\sigma_4(x))\\
=&(g(ux+c\sigma_1(v)\sigma_2(y)),h(\sigma_3(u)y+v\sigma_4(x))b^{-1})\\
=&(g(ux)+g(c)g(\sigma_1(v))g(\sigma_2(y)),[h(\sigma_3(u))h(y)+h(v)h(\sigma_4(x))]b^{-1}).
\end{align*}
By (\ref{Equation 1}) and (\ref{Equation 2}), we obtain equality and thus $G$ is an $F$-algebra isomorphism.
\end{proof}

%\begin{theorem}\label{ConstructingIsomorphismsKtoL}
%Let $K,L$ be two isomorphic finite field extensions of $F$ and $g,h:K\to L$ be $F$-isomorphisms. Let $b\in L^{\times}$. Let $A=\text{Cay}(K,c,\sigma_1,\sigma_2,\sigma_3,\sigma_4)$ and $B=\text{Cay}(L,g(c)\phi_1(b)\phi_2(b),\phi_1,\phi_2,\phi_3,\phi_4)$ with
%\begin{align}
%\phi_i&=g\circ\sigma_i\circ h^{-1} \mbox{ for } i=1,2,\\
%\phi_i&=h\circ\sigma_i\circ g^{-1} \mbox{ for } i=3,4.
%\end{align}
%Then the map $G:A\to B,\quad G(u,v)=(g(u),h(v)b^{-1})$ defines an $F$-algebra isomorphism.
%\end{theorem}
%\begin{proof}
%Clearly $G$ is $F$-linear, additive and bijective. It only remains to show that $G$ is multiplicative; that is, $$G((u,v)(x,y))=G(u,v)G(x,y)$$ for all $u,v,x,y\in K$. First we have \begin{align*}
%G(u,v)G(x,y)=&(g(u),h(v)b^{-1})(g(x),h(y)b^{-1})\\
%=&(g(u)g(x)+g(c)\phi_1(b)\phi_2(b)\phi_1(h(v)b^{-1})\phi_2(h(y)b^{-1}),\phi_3(g(u))h(y)b^{-1}+h(v)b^{-1}\phi_4(g(x)))\\
%=&(g(ux)+g(c)\phi_1(h(v))\phi_2(h(y)),[\phi_3(g(u))h(y)+h(v)\phi_4(g(x))]b^{-1}).
%\end{align*}
%Further, we have \begin{align*}
%G((u,v)(x,y))=&G(ux+c\sigma_1(v)\sigma_2(y),\sigma_3(u)y+v\sigma_4(x))\\
%=&(g(ux+c\sigma_1(v)\sigma_2(y)),h(\sigma_3(u)y+v\sigma_4(x))b^{-1})\\
%=&(g(ux)+g(c)g(\sigma_1(v))g(\sigma_2(y)),[h(\sigma_3(u))h(y)+h(v)h(\sigma_4(x))]b^{-1}).
%\end{align*}
%By (1) and (2), we obtain equality and thus $G$ is an $F$-algebra isomorphism.
%\end{proof}

\begin{corollary}\label{ConstructingIsomorphismsKtoK}
Let $g,h\in Aut_F(D)$ (resp. $Aut_F(K)$) and $b\in F^{\times}$ (resp. $b\in K^{\times})$. Let $B_D=\text{Cay}(D,g(c)b^2,\phi_1,\phi_2,\phi_3,\phi_4)$ (resp. $B_K=\text{Cay}(K,g(c)\phi_1(b)\phi_2(b),\phi_1,\phi_2,\phi_3,\phi_4)$ for some $b\in K^{\times}$). If 
\begin{align*}
\phi_i&=g\circ\sigma_i\circ h^{-1} \mbox{ for } i=1,2,\\
\phi_i&=h\circ\sigma_i\circ g^{-1} \mbox{ for } i=3,4,
\end{align*}
then the map $G:A\to B,\quad G(u,v)=(g(u),h(v)b^{-1})$ defines an $F$-algebra isomorphism.
\end{corollary}

\begin{corollary} Every generalised Dickson algebra $A_D=\text{Cay}(D,c,\sigma_1,\sigma_2,\sigma_3,\sigma_4)$ is isomorphic to an algebra of the form $\text{Cay}(D,c,\sigma_1',\sigma_2',\sigma_3', id)$ (analogously for the algebras $A_K$).
\end{corollary}
\begin{proof}
Consider the map $G:D\oplus D\to D\oplus D$ defined by $G(u,v)=(u,\sigma_4^{-1}(v))$. By Theorem \ref{ConstructingIsomorphismsDtoD'}, this yields the isomorphism $\text{Cay}(D,c,\sigma_1,\sigma_2,\sigma_3,\sigma_4)\cong \text{Cay}(D,c,\sigma_1\circ\sigma_4,\sigma_2\circ\sigma_4,\sigma_4^{-1}\circ\sigma_3,id).$
\end{proof}

\begin{remark} If $Comm(A_D)\neq F$ or $Comm(A_K)\not\subset K$, then $\sigma_1=\sigma_2$ and $\sigma_3=\sigma_4$ by Lemma \ref{Commutator of Cay(K)}. Via the map $G(u,v)=(u,\sigma_3^{-1}(v))$, Corollary \ref{ConstructingIsomorphismsKtoK} yields that every such algebra is isomorphic to the generalisation of commutative Dickson algebras as defined in \cite{Thompson19}.
\end{remark}

In certain cases, the maps defined in Theorem \ref{ConstructingIsomorphismsDtoD'} and Corollary \ref{ConstructingIsomorphismsKtoK} are the only possible isomorphisms between two algebras constructed via our generalised Cayley-Dickson doubling:

\begin{theorem}\label{IsomorphismRestrictToKtoL}
Let $A_K=\text{Cay}(K,c,\sigma_1,\sigma_2,\sigma_3,\sigma_4)$ and $B_L=\text{Cay}(L,c',\phi_1,\phi_2,\phi_3,\phi_4)$. Suppose that $G:A_K\to B_L$ is an isomorphism that restricts to an isomorphism $g:K\to L$. Then $G$ is of the form $G(x,y)=(g(x),h(y)b)$ for some isomorphism $h:K\to L$ such that $\phi_i\circ h=g\circ\sigma_i$ for $i=1,2$ and $\phi_i\circ g=h\circ\sigma_i$ for $i=3,4$ and some $b\in L^{\times}$ such that $g(c)=c'\phi_1(b)\phi_2(b).$
\end{theorem}
\begin{proof}
Suppose $G$ is an isomorphism from $A_K$ to $B_L$ such that $G\!\mid_K=g:K\to L$ is an isomorphism. Then for all $x\in K$, we have $G(x,0)=(g(x),0).$ Let $G(0,1)=(a,b)$ for some $a,b\in L$. As $G$ is multiplicative, this yields\begin{align*}
G(x,y)=&G(x,0)+G(\sigma_3^{-1}(y),0)G(0,1)\\
=&(g(x),0)+(g(\sigma_3^{-1}(y)),0)(a,b)\\
=&(g(x)+g(\sigma_3^{-1}(y))a,\phi_3(g(\sigma_3^{-1}(y)))b),
\end{align*}
and 
\begin{align*}
G(x,y)=&G(x,0)+G(0,1)G(\sigma_4^{-1}(y),0)\\
=&(g(x),0)+(a,b)(g(\sigma_4^{-1}(y)),0)\\
=&(g(x)+g(\sigma_4^{-1}(y))a,b\phi_4(g(\sigma_4^{-1}(y)))).
\end{align*}
It follows that either $\phi_3\circ g\circ\sigma_3^{-1}=\phi_4\circ g\circ\sigma_4^{-1}$ or $b=0$. However, if $b=0$ this would imply that $G$ was not surjective, which is a contradiction to the assumption that $G$ is an isomorphism. Thus it follows that $\phi_3\circ g\circ\sigma_3^{-1}=\phi_4\circ g\circ\sigma_4^{-1}.$ Additionally, we have either $g\circ\sigma_3^{-1}=g\circ\sigma_4^{-1}$ or $a=0$.\\
Consider $G((0,1)^2)=G(0,1)^2$. This gives $(a^2+c'\phi_1(b)\phi_2(b),\phi_3(a)b+b\phi_4(a))=(g(c),0).$ As we have established that $b\neq 0$, this implies that $\phi_3(a)=-\phi_4(a).$ If $a\neq 0$, we obtain $g\circ\sigma_3^{-1}=g\circ\sigma_4^{-1}$. Substituting this into the condition $\phi_3\circ g\circ\sigma_3^{-1}=\phi_4\circ g\circ\sigma_4^{-1}$, we conclude that $\phi_3=\phi_4$. This contradicts $\phi_3(a)=-\phi_4(a)$. Thus we must in fact have $a=0$ and $G(x,y)=(g(x),h(y)b)$ where $h=\phi_3\circ g\circ\sigma_3^{-1}$ and $g(c)=c'\phi_1(b)\phi_2(b)$. Computing $G(u,v)G(x,y)=G((u,v)(x,y))$ gives the remaining conditions.
\end{proof}

This proof does not hold when we consider the algebras $A_D$, as we rely heavily on the commutativity of $K$.

\begin{corollary}\label{IsomorphismRestrictToK}
%Let $A_K=\text{Cay}(K,c,\sigma_1,\sigma_2,\sigma_3,\sigma_4)$ and $B_K=\text{Cay}(K,c',\phi_1,\phi_2,\phi_3,\phi_4)$.
 Suppose that $G:A_K\to B_K$ is an isomorphism that restricts to an automorphism $g$ of $K$. Then $G$ is of the form $G(x,y)=(g(x),h(y)b)$ for $g,h\in Aut_F(K)$ such that $\phi_i\circ h=g\circ\sigma_i$ for $i=1,2$ and $\phi_i\circ g=h\circ\sigma_i$ for $i=3,4$ and some $b\in K^{\times}$ such that $g(c)=c'\phi_1(b)\phi_2(b).$
\end{corollary}
If $\text{Nuc}_l(A)=\text{Nuc}_l(B)=K$, all isomorphisms from $A\to B$ must restrict to an automorphism of $K$; similar considerations are true for restrictions to the middle and right nuclei. It follows that we can determine precisely when two such algebras are isomorphic by Corollary \ref{IsomorphismRestrictToK}.
%
%\begin{remark} There is a corollary to this when $Aut_F(K)=\langle\sigma\rangle$ which allows us to construct a lower bound on how many non-isomorphic division algebras there are with left, middle, and right nuclei equal to $K$; this is written up on paper at the moment but will be typed up shortly. The idea is that, after fixing a scalar $c$, every division algebra we construct is isomorphic to exactly $\left|Aut_F(K)\right|-1$ other division algebras of the same form, so we can partition the possible $n^4$ algebras into $n^3$ isomorphism classes. This was done in full for $[K:F]=2$ in the next subsection and should be easy to replicate for $[K:F]=p$, for some prime. An oversight to the work that was done, however, is that we never considered the possibility that $\sigma_3(c)=\sigma_4(c)$. In addition to this, we haven't looked at what happens when we take different choices of $c$. There are some comments to be made on it but we need to pursue it further. 
%\end{remark}

\begin{corollary}\label{IsomorphismRestrictToKCommutes}
%Let $A_K=\text{Cay}(K,c,\sigma_1,\sigma_2,\sigma_3,\sigma_4)$ and $B_K=\text{Cay}(K,c',\phi_1,\phi_2,\phi_3,\phi_4)$. 
Suppose that $G:A_K\to B_K$ is an isomorphism that restricts to an automorphism $g$ of $K$. If $\sigma_i=\phi_i=id$ for any $i=1,2,3,4$, $G$ must be of the form $$G(x,y)=(g(x),g(y)b)$$ for $g\in Aut_F(K)$ such that $\phi_i\circ g=g\circ\sigma_i$ for $i=1,2$ and $\sigma_i\circ g=g\circ\phi_i$ for $i=3,4$ and some $b\in K^{\times}$ such that $g(c)=c'\phi_1(b)\phi_2(b).$
\end{corollary}
\begin{proof}
From Theorem \ref{IsomorphismRestrictToK}, we see that $\phi_i\circ h=g\circ\sigma_i.$ If $\phi_i=\sigma_i=id$ for some $i=1,2,3,4$, we conclude that $g=h$ and the result follows.
\end{proof}

\begin{corollary}\label{IsomorphismRestrictToKN}
%Let $A_K=\text{Cay}(K,c,\sigma_1,\sigma_2,\sigma_3,\sigma_4)$ and $B_K=\text{Cay}(K,c',\phi_1,\phi_2,\phi_3,\phi_4)$ and 
Suppose that $G:A_K\to B_K$ is an isomorphism that restricts to an automorphism of $K$. If $K$ is a separable extension of $F$, we must have $N_{K/F}(cc'^{-1})=N_{K/F}(b^2)$ for some $b\in K^{\times}$.
\end{corollary}
\begin{proof}
Suppose $G:A_K\to B_K$ is an isomorphism that restricts to an automorphism of $K$. By Theorem \ref{IsomorphismRestrictToK}, we have $g(c)=c'\phi_1(b)\phi_2(b).$ Applying norms to both side, we obtain $$N_{K/F}(g(c))=N_{K/F}(c'\phi_1(b)\phi_2(b)).$$ As $K$ is a separable extension of $F$, it follows that $N_{K/F}(g(x))=N_{K/F}(x)$ for all $x\in K$, $g\in Aut_F(K)$. This yields $N_{K/F}(c)=N_{K/F}(c'b^2).$ As $c'\in K^{\times}$ and $N_{K/F}$ is multiplicative, we conclude that $N_{K/F}(cc'^{-1})=N_{K/F}(b^2)$.
\end{proof}

\begin{example} Let $F=\mathbb{Q}_p$ ($p\neq 2$) and $K$ be a separable extension of $\mathbb{Q}_p$. It is well known that $(\mathbb{Q}_p^{\times})^2/\mathbb{Q}_p=\lbrace [1],[u],[p],[up]\rbrace$ for some $u\in\mathbb{Z}_p\setminus \mathbb{Z}_p^2$. If $N_{K/F}(c)$ and $N_{K/F}(c')$ do not lie in the same coset of $(\mathbb{Q}_p^{\times})^2/\mathbb{Q}_p$, there does not exist an isomorphism that restricts to $K$ such that $\text{Cay}(K,c,\sigma_1,\sigma_2,\sigma_3,\sigma_4)\cong \text{Cay}(K,c',\phi_1,\phi_2,\phi_3,\phi_4)$ by Corollary \ref{IsomorphismRestrictToKN}.
\end{example}

\subsection{Automorphisms}

%\begin{theorem}\label{Sufficient condition for automorphisms of K}
\begin{theorem}\label{CSA Construction Automorphisms}
Let $g,h\in Aut_F(D)$ (resp. $Aut_F(K)$) such that $g\circ f=f\circ h$ for $f=\sigma_1,\sigma_2,\sigma_3^{-1},\sigma_4^{-1}$ and let $b\in F^{\times}$ (resp. $b\in K^{\times}$) such that $g(c)=b^2c$ (resp. $g(c)=\sigma_1(b)\sigma_2(b)c$). Then the map $G:A\to A$ defined by $G(u,v)=(g(u), h(v)b)$ is an automorphism of $A_D$ (resp. $A_K$).
\end{theorem}
This is easily checked via some long calculations.

\begin{theorem}\label{GeneralAutomorphisms}
Suppose that at least one of $\text{Nuc}_l(A_K)$, $\text{Nuc}_m(A_K)$, $\text{Nuc}_r(A_K)$ is equal to $K$. Then $G:A_K\to A_K$ is an automorphism of $A_K$ if and only if $G$ has the form stated in Theorem \ref{CSA Construction Automorphisms}.
\end{theorem}
 \begin{proof}
 Let $A=A_K$. Suppose $G\in Aut_F(A)$ and $\text{Nuc}_l(A)=K$. As automorphisms preserve the nuclei of an algebra, $G$ restricted to $\text{Nuc}_l(A)$ must be an automorphism of $K$; that is, $G\mid_K= g\in Aut_F(K)$ and so we have $G(x,0)=( g(x),0)$ for all $x\in K$.\\
If $\text{Nuc}_l(A)\neq K$, by our assumptions one of $\text{Nuc}_m(A)$ or $\text{Nuc}_r(A)$ are equal to $K$. In either case, we can use an identical argument by restricting $G$ to $\text{Nuc}_m(A)$ or $\text{Nuc}_r(A)$ respectively. As automorphisms preserve the nuclei of an algebra, $G$ restricted to $\text{Nuc}_m(A)$ (respectively $\text{Nuc}_r(A)$) must be an automorphism of $K$. Let $G(0,1)=(a,b)$ for some $a,b\in K$. Then \begin{align*}
 G(x,y)=&G(x,0)+G(\sigma_3^{-1}(y),0)G(0,1)\\
=&( g(x)+ g\sigma_3^{-1}(y)a,\sigma_3 g\sigma_3^{-1}(y)b),
 \end{align*}
and also \begin{align*}
 G(x,y)=&G(x,0)+G(0,1)G(\sigma_4^{-1}(y),0)\\
=&( g(x)+ g\sigma_4^{-1}(y)a,\sigma_4 g\sigma_4^{-1}(y)b)
\end{align*}
for all $x,y\in K$. Hence we must have $ g\sigma_3^{-1}(y)a= g\sigma_4^{-1}(y)a$ for all $y\in K$, which implies either $\sigma_3=\sigma_4$ or $a=0$. Additionally we have $\sigma_3 g\sigma_3^{-1}(y)b=\sigma_4 g\sigma_4^{-1}(y)b$. If $b=0$, this would imply $G(x,y)=( g(x)+ g\sigma_4^{-1}(y)a,0)$, which is a contradiction as it implies $G$ is not surjective. Thus we must in fact have $\sigma_3 g\sigma_3^{-1}(y)=\sigma_4 g\sigma_4^{-1}(y)$ for all $y\in K$.\\
Now we consider $G((0,1)^2)=G(0,1)^2$. This gives $(a,b)(a,b)=( g(c),0),$ which implies \begin{align*}
a^2+c\sigma_1(b)\sigma_2(b)=& g(c),\\
\sigma_3(a)b+b\sigma_4(a)=& 0.
\end{align*}
If $\sigma_3\neq\sigma_4$, we already know that $a=0$. On the other hand if $\sigma_3=\sigma_4$, we obtain $2\sigma_3(a)b=0$. As $K$ has characteristic not 2 and $b\neq 0$, this implies $a=0$. In either case, we obtain $c\sigma_1(b)\sigma_2(b)= g(c)$ and $G(u,v)=( g(u), h(v)b),$ where $ h=\sigma_3\circ g\circ\sigma_3^{-1}=\sigma_4\circ g\circ\sigma_4^{-1}.$ We note that this definition of $ h$ implies that $ h\circ\sigma_3=\sigma_3\circ g$ and $ h\circ\sigma_4=\sigma_4\circ g$.\\
Finally we consider $G(u,v)G(x,y)=G((u,v)(x,y))$ for all $u,v,x,y\in K$. We obtain $( g(u), h(v)b)( g(x), h(y)b)=( g(uv+c\sigma_1(v)\sigma_2(y)), h(\sigma_3(u)y+v\sigma_4(x))b)$ which gives the equations \begin{align*}
c\sigma_1( h(v)b)\sigma_2( h(y)b)=& g(c) g(\sigma_1(v)\sigma_2(y)),\\
\sigma_3( g(u)) h(y)b+ h(y)\sigma_4( g(x))b=& h(\sigma_3(u)y+v\sigma_4(x))b.
\end{align*}
 As $ h\circ\sigma_3=\sigma_3\circ g$ and $ h\circ\sigma_4=\sigma_4\circ g$, the second equation holds for all $u,v,x,y\in K$. Substituting $ g(c)=c\sigma_1(b)\sigma_2(b)$ into the first equation, we obtain $\sigma_1(h(v))\sigma_2(h(y))= g(\sigma_1(v)) g(\sigma_2(y))$ for all $v,y\in K$. This implies $\sigma_1\circ h= g\circ\sigma_1$ and $\sigma_2\circ h= g\circ\sigma_2$. Hence if $G$ is an automorphism of $A$ we must have $G(u,v)=( g(u), h(v)b)$ for some $ g, h\in Aut_F(K)$, such that $ g\circ f=f\circ h$ for $f=\sigma_1,\sigma_2,\sigma_3^{-1},\sigma_4^{-1}$ and some $b\in K^{\times}$, such that $ g(c)=\sigma_1(b)\sigma_2(b)c.$
\end{proof}

%If we assume that $K$ is a cyclic extension of $F$ with $Aut_F(K)=\langle\sigma\rangle$, we get the special case as shown below:
 
\begin{corollary}\label{KDoublingAutGroup}
Suppose that at least one of $\text{Nuc}_l(A_K)$, $\text{Nuc}_m(A_K)$, $\text{Nuc}_r(A_K)$ is equal to $K$ and $Aut_F(K)=\langle\sigma\rangle$. Then $G:A_K\to A_K$ is an automorphism of $A_K$ if and only if $G(u,v)=(\sigma^i(u),\sigma^i(v)b)$ for some $i\in \mathbb{Z}$ and $b\in K^{\times}$ satisfying $\sigma^i(c)=c\sigma^{\alpha_2}(b)\sigma^{\beta_2}(b).$ 
\end{corollary}
%\begin{proof}
%As $Aut_F(K)=\langle\sigma\rangle$, we can write the multiplication in $A$ as $$(u,v)(x,y)=(uv+c\sigma^{\alpha_2}(v)\sigma^{\beta_2}(y),\sigma^{\alpha_3}(u)y+v\sigma^{\beta_4}(x))$$ for some $c\in K^{\times}$ and $\alpha_2,\beta_2,\alpha_3,\beta_4\in \mathbb{Z}$.\\
%Now suppose $G\in Aut_F(A)$. As $Aut_F(K)$ is cyclic and therefore abelian, Theorem \ref{GeneralAutomorphisms} gives the following:\\
%$G:A\to A$ is an automorphism of $A$ if and only if $$G(u,v)=( g(u), g(v)b)$$ for some $ g\in Aut_F(K)$ and $b\in K^{\times}$ such that $$ g(c)=\sigma^{\alpha_2}(b)\sigma^{\beta_2}(b)c.$$ Further, as $Aut_F(A)$ is cyclic, we must have $ g=\sigma^i$ for some $i\in \mathbb{Z}$. The corollary follows immediately from this.
%\end{proof}

In the case when doubling a central simple algebra, we obtain a partial generalisation of Theorem \ref{GeneralAutomorphisms}:

\begin{lemma}\label{CSAAutomorphismsScalars} Let $G\in Aut(A_D)$ be such that $G\!\mid_D=g\in Aut_F(D)$. Then there must exist some $a,b\in D$, $b\neq 0$, such that for all $y\in D$, $$a g\circ\sigma_4^{-1}(y)=g\circ\sigma_3^{-1}(y)a,$$ $$b\sigma_4\circ g\circ\sigma_4^{-1}(y)=\sigma_3\circ g\circ\sigma_3^{-1}(y)b.$$
\end{lemma}
\begin{proof}
Suppose $G\mid_D=g\in Aut_F(D)$. Then for all $x\in D$, we obtain $G(x,0)=(g(x),0)$. Let $G(0,1)=(a,b)$ for some $a,b\in D$. It now follows that \begin{align*}
 G(x,y)=&G(x,0)+G(\sigma_3^{-1}(y),0)G(0,1)\\
% =&(g(x),0)+(g\circ\sigma_3^{-1}(y),0)(a,b)\\
=&( g(x)+ g\circ\sigma_3^{-1}(y)a,\sigma_3\circ g\circ\sigma_3^{-1}(y)b),
 \end{align*}
and also \begin{align*}
 G(x,y)=&G(x,0)+G(0,1)G(\sigma_4^{-1}(y),0)\\
 %=&(g(x),0)+(a,b)(g\circ\sigma_4^{-1}(y),0)\\
=&( g(x)+ ag\circ\sigma_4^{-1}(y),b\sigma_4\circ g\circ\sigma_4^{-1}(y)).
\end{align*}
Setting these two equivalent expressions for $G(x,y)$ equal to each other yields the result. Note that if $b=0$, $G$ would no longer be surjective, which would contradict our assumption that $G\in Aut(A_d)$. 
\end{proof}

\begin{theorem}\label{CSAAutomorphisms3=4} Let $G\in Aut(A_D)$ be such that $G\!\mid_D=g\in Aut_F(D)$. If $\sigma_3=\sigma_4$, then $G:A_D\to A_D$ must have the form as stated in Theorem \ref{CSA Construction Automorphisms}.
\end{theorem}
\begin{proof}
Suppose $G\!\mid_D=g\in Aut_F(D)$. Substituting $\sigma_3=\sigma_4$ into Lemma \ref{CSAAutomorphismsScalars}, we see that $G(0,1)=(a,b)$ for some $a,b\in D$ such that $$a g\circ\sigma_3^{-1}(y)=g\circ\sigma_3^{-1}(y)a,$$ $$b\sigma_3\circ g\circ\sigma_3^{-1}(y)=\sigma_3\circ g\circ\sigma_3^{-1}(y)b.$$ This is satisfied for all $y\in D$ if and only if $a,b\in F$ and so $G(x,y)=(g(x)+g\circ\sigma_3^{-1}(y)a,\sigma_3\circ g\circ\sigma_3^{-1}(y)b).$ The remainder of this proof follows almost exactly the same to Theorem \ref{GeneralAutomorphisms}:

Now we consider $G((0,1)^2)=G(0,1)^2$. This gives $(a,b)(a,b)=( g(c),0),$ which implies \begin{align*}
a^2+c\sigma_1(b)\sigma_2(b)=& g(c)\\
\sigma_3(a)b+b\sigma_4(a)=0.
\end{align*}
As $a,b\in F$, the second equation is equivalent to $2ab=0$. As $F$ has characteristic not 2, this implies $a=0$ or $b=0$. If $b=0$, $G$ would not be surjective, which contradicts our assumption that $G$ is an isomorphism. Thus we must have $a=0$ and so we obtain $g(c)=cb^2$ and $G(u,v)=( g(u), h(v)b),$ where $ h=\sigma_3\circ g\circ\sigma_3^{-1}.$ We note that this definition of $ h$ implies that $ h\circ\sigma_3=\sigma_3\circ g$.\\
Finally we consider $G(u,v)G(x,y)=G((u,v)(x,y))$ for all $u,v,x,y\in D$. We obtain $( g(u), h(v)b)( g(x), h(y)b)=( g(uv+c\sigma_1(v)\sigma_2(y)), h(\sigma_3(u)y+v\sigma_4(x))b),$ which gives the equations \begin{align*}
c\sigma_1( h(v)b)\sigma_2( h(y)b)=& g(c) g(\sigma_1(v)\sigma_2(y)),\\
\sigma_3( g(u)) h(y)b+ h(y)\sigma_4( g(x))b=& h(\sigma_3(u)y+v\sigma_4(x))b.
\end{align*} 
After substituting $cb^2=g(c)$, we conclude that this is satisfied for all $x,y,u,v\in D$ if and only if we have $\sigma_1\circ h=g\circ \sigma_1$ and $\sigma_2\circ h=g\circ \sigma_2$.
\end{proof}

\printbibliography

\end{document}